\documentclass[12pt]{amsart}
\usepackage{amsmath}
\usepackage{amssymb}
\usepackage[pdftex]{graphicx}
\usepackage{amsthm}
\usepackage{mathscinet}
\usepackage{mathrsfs}
\usepackage[utf8]{inputenc} 
\usepackage{tikz}
\usetikzlibrary{arrows ,scopes}

\DeclareMathOperator{\U}{U}
\DeclareMathOperator{\LL}{L}

\title[An Ansatz for Hyperk\"{a}hler $8$-Manifolds with Symmetry]{An Ansatz for Hyperk\"{a}hler $8$-Manifolds with two Commuting Rotating Killing Fields}
\author{Joseph Malkoun}
\address{Department of Mathematics and Statiscs\\
Notre Dame University-Louaize\\
Lebanon}
\email{joseph.malkoun@ndu.edu.lb}
\date{\today}

\newtheorem{theorem}{Theorem}[section]

\theoremstyle{remark}
\newtheorem{remark}[theorem]{Remark}

\theoremstyle{definition}

\begin{document}

\maketitle
\begin{abstract}
We consider a hyperk\"{a}hler $8$-manifold admitting either a $\U(1) \times \mathbb{R}$, or a $\U(1) \times \U(1)$ action, where the first factor preserves $g$ and $I$, and acts on $\omega_2+i\omega_3$ by multiplying it by itself, while the 
second factor preserves $g$ and acts triholomorphically. Such data can be reduced to a single function $H$ of two complex variables and two real variables satisfying $6$ equations of Monge-Ampere type, which can be compactly written down using a Poisson bracket. 
\end{abstract}

\section{Introduction}

In recent years, both mathematicians (for example \cite{Haydys-2008}, \cite{Hitchin-2013} and \cite{ACM-2013}) 
and physicists (for example \cite{Butter_Kuzenko-2011}) became interested in hyperk\"{a}hler manifolds with so-called 
``rotating'' Killing fields: these are Killing fields whose flow preserves a complex structure in the $S^2$ of complex structures 
of a hyperk\"{a}hler metric, say $I$, and rotates the other two, $J$ and $K$, in the plane spanned by $J$ and $K$. 
In dimension $4$, those already had been studied by 
authors such as Boyer and Finley in \cite{Boyer_Finley-1982}, and their symmetry reduction leads to the Boyer-Finley-LeBrun equation 
(see \cite{Boyer_Finley-1982} and the later work by LeBrun in \cite{LeBrun-1991}).

Consider flat quaternionic space in real dimension $8$. This can be described as $\mathbb{C}^4$ with coordinates 
$(q^1,q^2,p_1,p_2)$, together with the K\"{a}hler form
\[ \omega_1 = \frac{i}{2}\left(\sum_{j=1}^2 dq^j \wedge d\bar{q}^j + \sum_{j=1}^2 dp_j \wedge d\bar{p}_j \right) \]
as well as the holomorphic symplectic form
\[ \omega_+ = \sum_{j=1}^2 dq^j \wedge dp_j \]
The function
\[ \Omega = \sum_{j=1}^2 |q^j|^2 + \sum_{j=1}^2 |p_j|^2 \]
is a K\"{a}hler potential for $\omega_1$; in other words
\[ \omega_1 = \frac{i}{2}\partial \bar{\partial} \Omega \]
(note the unconventional factor of $1/2$). Let $t$ and $c$ be real variables, so that 
$(e^{it}, c) \in \U(1) \times \mathbb{R}$. The action $\LL$ of $\U(1) \times \mathbb{R}$ on 
$\mathbb{C}^4$ defined by
\[ \LL_{(e^{it},c)}(q^1, q^2, p_1, p_2)^T = (q^1+c, q^2, e^{it} p_1, e^{it} p_2)^T \]
preserves the complex structure of $\mathbb{C}^4$ which we denote by $I$, as well as $\omega_1$, and 
its effect on $\omega_+$ is as follows:
\[ \LL_{(e^{it},c)}^*(\omega_+) = e^{it} \omega_+ \]
We let
\begin{equation} \begin{array}{cl} u &= q^1 + \bar{q}^1 \\
v &= i(\bar{q}^1 - q^1) \\
q &= q^2 \\
\zeta &= p_2/p_1 \\
\rho &= \ln(|p_1|^2) \\
\theta &= i (\ln(\bar{p}_1) - \ln(p_1)) \end{array} \label{new-coordinates}
\end{equation}
We note that, in the new coordinates $(u,v,q,\rho,\theta,\zeta)$,  $\Omega$ is independent of $\theta$. It 
is not independent of $u$ though. But $\Omega$ can be replaced by $\Omega + F + \bar{F}$, where $F$ is a 
holomorphic function on $\mathbb{C}^4$. We replace $\Omega$ by the following
\[ \Omega' = \Omega -\frac{1}{2} (q^1)^2 - \frac{1}{2} (\bar{q}^1)^2 \]
When expressed in the new coordinates, $\Omega'$ becomes independent of both $\theta$ and $u$, and is 
equal to the following function
\[ H(q,\zeta, v, \rho) = \frac{1}{2} v^2 + |q|^2 + e^{\rho}(1+|\zeta|^2) \]
We introduce the following bracket for pairs of functions of $(q,\zeta,v,\rho)$, defined by the following bivector
\begin{align} \{-,-\} = e^{-\rho}(i \partial_v \wedge \partial_{\rho} + i \zeta \partial_{\zeta} \wedge \partial_v + 
  \partial_{\zeta} \wedge \partial_q) \label{Poisson} \end{align}
It is easy to check that the Schouten-Nijenhuis bracket of this bivector with itself vanishes; in other words, 
the bracket is Poisson. Then it is straightforward to check that the following equations hold:
\begin{align}
\{ H_{\rho},iH_v \} &= 1 \label{main-eq-1} \\
\{ H_{\bar{\zeta}}, H_{\bar{q}} \} &= 1 \label{main-eq-2} \\
\{ H_{\rho}, H_{\bar{q}} \} &= \bar{\zeta} \label{main-eq-3} \\
\{ iH_v, H_{\bar{\zeta}} \} &= 0 \label{main-eq-4} \\
\{ H_{\rho}, H_{\bar{\zeta}} \} &= 0 \label{main-eq-5} \\
\{ iH_v, H_{\bar{q}} \} &= 0 \label{main-eq-6}
\end{align}
We claim that this is, in a sense, the general case for such types of actions. More precisely, we have
\begin{theorem} Let $H(q,\zeta, v, \rho)$ be a real-valued function on $\mathbb{C}^2 \times \mathbb{R}^2$ 
which satisfies equations (\ref{main-eq-1})-(\ref{main-eq-6}). Then, on 
$\mathbb{C}^2 \times \mathbb{R}^2 \times \U(1) \times \mathbb{R}$, with coordinates 
$(q, \zeta, v, \rho, \lambda, u)$, we have the natural projection
\[ \pi: \mathbb{C}^2 \times \mathbb{R}^2 \times \U(1) \times \mathbb{R} \to \mathbb{C}^2 \times \mathbb{R}^2 \]
mapping $(q, \zeta, v, \rho, \lambda, u)$ to $(q, \zeta, v, \rho)$. We introduce new coordinates on 
$\mathbb{C}^2 \times \mathbb{R}^2 \times \U(1) \times \mathbb{R} \simeq \mathbb{C}^4$:
\begin{align*} q^1 &= \frac{1}{2}(u+iv) \\
q^2 &= q \\
p_1 &= e^{\rho/2} \lambda \\
p_2 &= e^{\rho/2} \lambda \zeta
\end{align*}
Then, if we let $\Omega(q^1,q^2,p_1,p_2)$ be $\pi^*(H)$ expressed in the new coordinates $(q^i,p_i)$, and if 
we let
\begin{align*} \omega_1 &= \frac{i}{2} \partial \bar{\partial} \Omega \\
\omega_+ &= \sum_{j=1}^2 dq^j \wedge dp_j
\end{align*}
then $\mathbb{C}^4$ with its complex structure $I$, together with $\omega_1$ and $\omega_+$ is hyperk\"{a}hler 
having a $\U(1) \times \mathbb{R}$ action $\LL$
\[ \LL_{(e^{it},c)}(q^1, q^2, p_1, p_2)^T = (q^1+c, q^2, e^{it} p_1, e^{it} p_2)^T \]
which preserves $I$ and $\omega_1$, and such that
\begin{equation} \LL_{(e^{it},c)}^*(\omega_+) = e^{it} \omega_+ \label{action} \end{equation}
Conversely, any hyperk\"{a}hler $8$-dimensional manifold having a free $\U(1)\times \mathbb{R}$ action preserving $I$ and 
$\omega_1$ and acting on $\omega_+$ as in (\ref{action}) can be locally described by a K\"{a}hler potential function 
with respect to $I$ which is a real-valued function of (an open subset of) $\mathbb{C}^2 \times \mathbb{R}^2$ satisfying equations 
(\ref{main-eq-1})-(\ref{main-eq-6}).
\end{theorem}

\begin{proof}
We start by proving the converse. It can be shown that there exist holomorphic coordinates 
$(q^1,q^2,p_1,p_2)$ with respect to $I$ such that
\[ \omega_+ = \sum_{j=1}^2 dq^j \wedge dp_j \]
and the action $L$ takes the required form
\[ \LL_{(e^{it},c)}(q^1, q^2, p_1, p_2)^T = (q^1+c, q^2, e^{it} p_1, e^{it} p_2)^T \]
This is in a way an equivariant version of the celebrated Darboux theorem in symplectic geometry. We introduce the coordinates
\begin{align*}
u &= q^1 + \bar{q}^1 \\
v &= i(\bar{q}^1 - q^1) \\
q &= q^2 \\
r &= |p_1| \\
\theta &= \frac{i}{2} \ln(\bar{p}_1) - \frac{i}{2} \ln(p_1) \\
\zeta &= \frac{p_2}{p_1}
\end{align*}
The inverse coordinate transformations are given by
\begin{align*}
q^1 &= \frac{1}{2}(u+iv) \\
q^2 &= q \\
p_1 &= r e^{i\theta} \\
p_2 &= r e^{i\theta} \zeta
\end{align*}
The coordinate vector fields in the two coordinate systems are related by
\begin{align*}
\partial_{q^1} &= \partial_u - i \partial_v \\
\partial_{q^2} &= \partial_q \\
\partial_{p_1} &= e^{-i\theta} \left( -\frac{\zeta}{r} \partial_{\zeta} + \frac{1}{2} \partial_r -\frac{i}{2r} \partial_{\theta} \right) \\
\partial_{p_2} &= \frac{1}{r}e^{-i\theta} \partial_{\zeta}
\end{align*}
We introduce the following Poisson bracket for pairs of functions on $\mathbb{C}^4$:
\begin{align*} &\{-,-\}_{\text{original}} \\ 
&= \sum_{j=1}^2 \partial_{p_j} \wedge \partial_{q^j} \\
&= e^{-i\theta} \left( \left( -\frac{\zeta}{r} \partial_{\zeta} + \frac{1}{2} \partial_r -\frac{i}{2r} \partial_{\theta} \right) 
\wedge (\partial_u -i \partial_v) + \frac{1}{r} \partial_{\zeta} \wedge \partial_q \right)
\end{align*}
If $\Omega$ is a K\"{a}hler potential for $I$, then the equations that $\Omega$ must satisfy for the metric to be hyperk\"{a}hler are 
the following symplectic Monge-Ampere equations (see \cite{Calabi-1979}):
\begin{align}
\{\Omega_{\bar{p}_i}, \Omega_{\bar{q}^j} \}_{\text{original}} &= \delta^i_{j} \label{C} \\
\{\Omega_{\bar{q}^i}, \Omega_{\bar{q}^j} \}_{\text{original}} &= 0 \label{A} \\
\{\Omega_{\bar{p}_i}, \Omega_{\bar{p}_j} \}_{\text{original}} &= 0 \label{B}
\end{align}
The equations (\ref{main-eq-1})-(\ref{main-eq-6}) are the symmetry reductions of these equations. The K\"{a}hler potential 
$\Omega$ can be replaced with
\[ \Omega' = \Omega + F + \bar{F} \] 
where $F$ is a holomorphic function of the coordinates $q^1$, $q^2$, $p_1$ and $p_2$. We claim that there is a holomorphic 
function $F$ such that in the $(u,v,q,r,\theta,\zeta)$ coordinates, $\Omega'$ does not depend on $u$ nor $\theta$. This 
follows from the fact $\partial_u$ and $\partial_{\theta}$ commute, and are each real parts of holomorphic vector fields. Denote 
by $K$ the function of the coordinates $(q,\zeta,v,r) \in \mathbb{C}^2 \times \mathbb{R} \times \mathbb{R}_+$ whose pullback to 
$\mathbb{C}^2 \times \mathbb{R} \times \mathbb{R}_+ \times \mathbb{R} \times (\mathbb{R}/(2 \pi \mathbb{Z}))$ with 
coordinates $(q,\zeta,v,r,u,\theta)$ by the natural projection is equal to $\Omega'$.

We denote by $\{-,-\}$ the Poisson bracket defined by (\ref{Poisson}). Then equation (\ref{A}) yields equation (\ref{main-eq-6}), 
and equation (\ref{B}) yields equation (\ref{main-eq-5}). The first claim is straightforward. Let us prove the second claim. 
Equation (\ref{B}) implies:
\begin{align*}
&\left\{ - \frac{\bar{\zeta}}{r} K_{\bar{\zeta}} + \frac{1}{2} K_r, \frac{1}{r}K_{\bar{\zeta}} \right\} \\
&= \frac{i}{2r^2}(-\frac{\bar{\zeta}}{r}K_{\bar{\zeta}} 
  + \frac{1}{2} K_r)K_{v\bar{\zeta}} - \frac{i}{2r^2}(-\frac{\bar{\zeta}}{r}K_{v\bar{\zeta}}+\frac{1}{2} K_{vr})K_{\bar{\zeta}} \\
&= \frac{i}{4r^2}(K_r K_{v\bar{\zeta}}-K_{\bar{\zeta}}K_{vr})
\end{align*}
But we also have that
\begin{align*}
&\left\{ - \frac{\bar{\zeta}}{r} K_{\bar{\zeta}} + \frac{1}{2} K_r, \frac{1}{r}K_{\bar{\zeta}} \right\} \\
&= \frac{1}{2} \left\{ K_r, \frac{1}{r}K_{\bar{\zeta}} \right\} \\
&= -\frac{i}{4r^2}K_{rv}K_{\bar{\zeta}} + \frac{i}{4r^2} K_r K_{v\bar{\zeta}} + \frac{1}{2r^2} \left\{ rK_r, K_{\bar{\zeta}} \right\}
\end{align*}  
And then equation (\ref{main-eq-5}) follows, using
\begin{equation} \rho = \ln(r^2) \end{equation}
Using similar calculations, one can show that the remaining equations, namely (\ref{main-eq-1})-(\ref{main-eq-4}), 
follow using the system (\ref{C}). We provide a few details. System (\ref{C}) implies
\begin{align*}
\left\{ -\frac{\bar{\zeta}}{r}K_{\bar{\zeta}} + \frac{1}{2} K_r, K_{\bar{q}^j} \right\} &= \delta^1_j + \frac{i}{2r}(-\frac{\bar{\zeta}}{r}K_{\bar{\zeta}} 
    +\frac{1}{2} K_r) K_{v\bar{q}^j} \\
\left\{ \frac{1}{r} K_{\bar{\zeta}}, K_{\bar{q}^j} \right\} &= \delta^2_j + \frac{i}{2r^2} K_{\bar{\zeta}}K_{v\bar{q}^j}
\end{align*}
Equivalently, one can use the second equation to replace the first equation by a simpler one. One then gets the following simpler system:
\begin{align*}
\frac{1}{2} \left\{ K_r, K_{\bar{q}^j} \right\} &= \delta^1_j + \bar{\zeta} \delta^2_j +\frac{i}{4r} K_r K_{v\bar{q}^j} \\
\left\{ \frac{1}{r} K_{\bar{\zeta}}, K_{\bar{q}^j} \right\} &= \delta^2_j + \frac{i}{2r^2} K_{\bar{\zeta}}K_{v\bar{q}^j}
\end{align*}
The first equation above, with $j=1$, yields equation (\ref{main-eq-1}), and with $j=2$, yields equation (\ref{main-eq-3}). 
The second equation above yields equations (\ref{main-eq-4}) ($j=1$) and (\ref{main-eq-2}) ($j=2$).

The other direction of the proof is easier, and can be obtained by going ``backwards'' in the argument above.
\end{proof}

\begin{remark} Since the proof above depends only on the generating vector fields of the $\U(1) \times \mathbb{R}$ action, one 
may just as well consider a $\U(1) \times \U(1)$ action instead, where the first $\U(1)$ factor preserves $g$ and $I$ and 
rotates $\omega_2$ and $\omega_3$, and the second factor preserves $g$ and is triholomorphic. A similar 
result holds in this case. We present a non-trivial example of the latter type in the following section. \end{remark}

\begin{remark} We remark that, on $\mathbb{C}^2 \times \mathbb{C}^2$, for functions that are pullbacks of functions 
on $\mathbb{C}^2 \times \mathbb{R}^2$, or in other words, for functions that are independent of $\theta$ and $u$, the 
two Poisson brackets are related by
\[ \{-,-\} = \frac{1}{\bar{p}_1} \{-,-\}_{\text{original}} \]
(mod $\partial_{\theta}$ and $\partial_u$). This can be used to provide another quick proof for why $\{-,-\}$ is Poisson, 
using the fact that $\{-,-\}_{\text{original}}$ is a holomorphic Poisson bracket, and that $\bar{p}_1$ is antiholomorphic. \end{remark}

\section{Example: the Calabi metric on $T^*(\mathbb{C}P^2)$}

Using a local affine chart $(z^1,z^2)$ on an affine subset of $\mathbb{C}P^2$, and corresponding fibre coordinates $(w_1,w_2)$ on the 
cotangent bundle restricted to that affine subset, then the coordinates $(z^1,z^2,w_1,w_2)$ are local holomorphic Darboux coordinates 
for the natural holomorphic symplectic structure on the cotangent bundle $T^*(\mathbb{C}P^2)$. The Calabi hyperk\"{a}hler structure (see \cite{Calabi-1979}) has as 
$\omega_2+i\omega_3$ the natural holomorphic symplectic structure on $T^*(\mathbb{C}P^2)$, and as K\"{a}hler potential for $I$ (the natural 
complex structure of $T^*(\mathbb{C}P^2)$ the following function:
\begin{equation}
\Omega = \log(1+|\mathbf{z}|^2) + \sqrt{1+4t} - \log(1+\sqrt{1+4t}),
\end{equation}
with
\begin{equation}
t = (1+|\mathbf{z}|^2)(|\mathbf{w}|^2 + |\sum_{j=1}^2 z^j w_j|^2)
\end{equation}
where $\mathbf{z} = (z^1,z^2)^T$ and $|\mathbf{z}|^2 = |z^1|^2 + |z^2|^2$, and similarly for $\mathbf{w}$ 
and $|\mathbf{w}|^2$. 

Consider the (local) action of $\U(1) \times \U(1)$ given by
\begin{equation}
(e^{it},e^{i\theta}).(\mathbf{z},\mathbf{w}) = (e^{i\theta} \mathbf{z}, e^{i(t-\theta)} \mathbf{w})
\end{equation}
This action preserves $g$ and $I$, and its action on $\omega_2+i\omega_3$ simply multiplies it 
by $e^{it}$.

Restricting further to the open subset given by $z^1 \neq 0$ and $z^2 \neq 0$, we make use of the 
following coordinate substitutions
\begin{align*}
z^1 &= e^{i(q^1-q^2)} \\
z^2 &= e^{i(q^1+q^2)} \\
w_1 &= \frac{i}{2}(p_2-p_1)e^{i(q^2-q^1)} \\
w_2 &= -\frac{i}{2}(p_1+p_2)e^{-i(q^1+q^2)}
\end{align*}
One can check that in the new coordinates $(q^1,q^2,p_1,p_2)$, which are also Darboux for 
$\omega_2 + i \omega_3$, the action is of the ``standard'' form
\[ \LL_{(e^{it},e^{i\theta})}(q^1, q^2, p_1, p_2)^T = (q^1+\theta, q^2, e^{it} p_1, e^{it} p_2)^T \]
Using then the coordinates $(u,v,q,r,\theta,\zeta)$ given by (\ref{new-coordinates}), we get that 
\begin{align*}
&|\mathbf{z}|^2 = e^{-v}(e^{i(\bar{q}-q)} + e^{-i(\bar{q}-q)}) \\
&|\mathbf{w}|^2 = \frac{1}{4} e^{\rho+v}(|1-\zeta|^2 e^{-i(\bar{q}-q)}+|1+\zeta|^2 e^{i(\bar{q}-q)}) \\
&|\sum_{j=1}^2 z^j w_j |^2 = e^{\rho}
\end{align*}
We note that, in the new coordinates $(u,v,q,r,\theta,\zeta)$, $\Omega$ is independent of $u$ and $\theta$, and the action 
is in standard form, so our theorem applies, and we have that, as a function of $(v,\rho,q,\zeta)$, with
\[ \rho = \log(r^2) \]
the function $\Omega$ becomes a function $H$ which satisfies the $6$ equations (\ref{main-eq-1})-(\ref{main-eq-6}).

\section{Conclusion}
While the equations we get (equations (\ref{main-eq-1})-(\ref{main-eq-6})) are 
difficult to solve, the fact that they can be neatly written down using the 
Poisson bracket (\ref{Poisson}) is quite interesting, in the author's opinion. 

\def\cprime{$'$}
\providecommand{\bysame}{\leavevmode\hbox to3em{\hrulefill}\thinspace}
\providecommand{\MR}{\relax\ifhmode\unskip\space\fi MR }
\providecommand{\MRhref}[2]{%
  \href{http://www.ams.org/mathscinet-getitem?mr=#1}{#2}
}
\providecommand{\href}[2]{#2}

\end{document}